\documentclass[12pt]{amsart}

\usepackage{graphicx}
\usepackage{hyperref}
\usepackage{url}

\usepackage{mathrsfs}
\usepackage{dsfont}
\usepackage{diagrams}

\usepackage{amsthm}

\theoremstyle{plain}
\newtheorem{thm}{Theorem}[section]
\newtheorem{lem}[thm]{Lemma}

\newtheorem{lem-defn}[thm]{Lemma-Definition}
\newtheorem{prop-defn}[thm]{Proposition-Definition}

\newtheorem*{thm*}{Theorem}
\newtheorem*{lem*}{Lemma}
\newtheorem*{prop*}{Proposition}
\newtheorem*{cor*}{Corollary}

\theoremstyle{definition}
\newtheorem{defn}[thm]{Definition}

\newtheorem*{defn*}{Definition}
\newtheorem*{exa*}{Example}
\newtheorem*{prob*}{Problem}

\theoremstyle{remark}
\newtheorem{rmk}[thm]{Remark}

\newtheorem*{rmk*}{Remark}
\newtheorem*{claim*}{Claim}
\usepackage{mathrsfs}
\usepackage{amssymb}
\usepackage{dsfont}
\usepackage{verbatim}
\usepackage{url}

\def\Z{\mathbb {Z}}

\def\R{\mathbb {R}}
\def\C{\mathbb {C}}

\def\jq{\mathbf {j}}
\def\kq{\mathbf {k}}
\def\mod{\textrm {mod}}

\def\from{\colon}

\DeclareMathOperator{\abl}{Abl}

\DeclareMathOperator{\gir}{Girth}

\DeclareMathOperator{\Mor}{Mor}

\DeclareRobustCommand
  \rddots{\mathinner{\mkern1mu\raise\p@
    \vbox{\kern7\p@\hbox{.}}\mkern2mu
    \raise4\p@\hbox{.}\mkern2mu\raise7\p@\hbox{.}\mkern1mu}}

\newcommand\xhookrightarrow[2][]{\ext@arrow 0062{\hookrightarrowfill@}{#1}{#2}}
\def\hookrightarrowfill@{\arrowfill@\lhook\relbar\rightarrow}

\newcommand\FF{\mathbb{F}}

\DeclareMathAlphabet{\mathcal}{OMS}{cmsy}{m}{n}

\newcommand\field\FF

\newcommand\real\R
\newcommand\complex\C
\newcommand\integers\Z

\newcommand{\red}{{\rm red}}

\newcommand{\girth}{{\rm girth}}

\begin{document}
\title{Abelian Girth and Girth}
\author{Joel Friedman}
\address{Department of Computer Science, The University of British Columbia, 2366 Main Mall, V6T 1Z4,  Vancouver BC, Canada, and Department of Mathematics, The University of British Columbia, 1984 Mathematics Road, V6T 1Z2,  Vancouver BC, Canada}
\email{jf@cs.ubc.ca or jf@math.ubc.ca}
\thanks{Research of all authors partly supported by an NSERC Discovery Grants.}

\author{Alice Izsak}
\address{Department of Computer Science, The University of British Columbia, 2366 Main Mall, V6T 1Z4,  Vancouver BC, Canada}
\email{aizsak@cs.ubc.ca or aizsak@gmail.com}

\author{Lior Silberman}
\address{Department of Mathematics, The University of British Columbia, 1984 Mathematics Road, V6T 1Z2,  Vancouver BC, Canada}
\email{lior@math.ubc.ca}
\urladdr{http://www.math.ubc.ca/~lior}


\begin{abstract}
We show that the abelian girth of a graph is at least three times
its girth.
We prove an analogue of the Moore bound for
the abelian girth of regular graphs, 
where the degree of the graph is fixed and the number of vertices
is large.
We conclude that one could try to improve the Moore bound for
graphs of fixed degree and many vertices by trying to improve
its analogue concerning the abelian girth.
\end{abstract}

\maketitle
\tableofcontents

\section{Introduction}

In the work \cite{friedman_hnc_memoirs} the first named author introduced the
\emph{abelian girth} of a graph, a graph invariant, and used it in a proof of
the Hanna Neumann Conjecture.  In this paper we study the abelian girth
and its relationship to the ordinary girth and to the volume bound on girth
known as the Moore bound.
Specifically, we give some evidence to show that bounding the abelian
girth from above is a possible approach to improving the Moore bound
for regular graphs of fixed degree with a large number of vertices.
Let us make this more precise.

The \emph{girth} of a graph is the length of the shortest non-trivial cycle
in the graph.  
If $G_{n,d}$ is any $d$-regular graph on $n$ vertices, it is known that
for fixed $d$ and large $n$ we have
\begin{equation}\label{eq:Moore_regular}
{\rm Girth}(G_{n,d}) \le 2 \log_{d-1}n + o_n(1),
\end{equation}
and we are interested to know if the factor of $2$ can be improved upon.
This bound follows from the {\em Moore bound} (see \cite{erdos66}.)
Although the Moore bound for regular graphs is very easy to prove,
there has been only slight improvements to it in the last 50 years---only
an additive constant of one or two. In particular, 
the multiplicative factor of $2$ in
\eqref{eq:Moore_regular}
is the best factor known to date.
Graphs of large girth have numerous applications 
(\cite{alon07,thorup,mackay96}) and much has been written on the
girth and the Moore bound (see \cite{amit}).

The {\em abelian girth} of a graph, $G$, denoted $\abl(G)$, is the 
girth of its {\em universal abelian cover}, or equivalently the shortest
length of a non-trivial, closed, non-backtracking walk that traverses
each edge the same number of times in each direction.
The abelian girth is
important in sheaf theory on graphs and the first proof of the Hanna Neumann
Conjecture \cite{friedman_hnc_memoirs};
furthermore, in a forthcoming paper we will show the abelian girth is
strongly related to what we call
``gapped'' sheaves on a graph
\cite{fis2}.
In this paper we show that there is an analogue of the Moore bound for the
abelian girth, and improving this analogue would improve the Moore bound.
Furthermore, we provide evidence that such an improvement may be possible.

Specifically, we show that
\begin{equation}\label{eq:main1}
\girth(G) \leq 3\abl(G),
\end{equation}
for any graph, $G$.  Second, we show that there is an argument analogous
to the Moore bound that shows that for fixed $d$ we have
\begin{equation}\label{eq:main2}
\abl(G_{n,d}) \le 6\log_{d-1}n + o_n(1) 
\end{equation}
for any $d$-regular graph on $n$ vertices, $G_{n,d}$; we remark
that the proof of this theorem is not as immediate as that of the Moore bound.
It follows that any improvement to the factor of $6$
in \eqref{eq:main2} would give an improvement to the factor of $2$
in \eqref{eq:Moore_regular}.

Ideally we would show that all known explicit constructions of families
of $d$-regular graphs have abelian girth at most
$c\log_{d-1}n $ for some $c<6$. In this paper,
we will focus on the only family of $d$-regular graphs with $d$ fixed
which has girth greater than $\log_{d-1}n $.
It is known that the factor of $2$ in \eqref{eq:Moore_regular} cannot
be less than $4/3$, at least for certain $d$: indeed,
\cite{LPS} constructs graphs, $X^{p,q}$, for primes $p,q\equiv 1 \pmod{4}$,
that are $d=p+1$ regular on $n=q(q^2+1)$ vertices for which 
$$
{\rm girth}(X^{p,q}) = (4/3) \log_{d-1} n + o_n(1)
$$
for fixed $d=p+1$ and large $n=q(q^2+1)$. Margulis \cite{margulisgirth}
independently constructed similar graphs with girth asymptotically as large as
that of the graphs from \cite{LPS}.
Furthermore there are
no known families of graphs which improve on the above $4/3$; in fact,
for general $d$, the best girth lower bound is $4/3$ replaced with $1$,
by choosing a random graph and slightly modifying it \cite{ellislinial}, \cite{esachs}.
To show that it is plausible to improve on the factor of $6$ in
\eqref{eq:main2}, we will show that
\begin{equation}\label{eq:main4}
\abl(X^{p,q}) \le  (16/3) \log_{d-1} n + o_n(1) 
\end{equation}
which suggests that there is room for the factor of $6$ to decrease.
We conjecture that the $16/3$ can be replaced with $4$, for reasons
we shall explain later.
We are unaware of any other graph constructions in the literature which improves
upon the $16/3$ above.

To prove \eqref{eq:main1}, we
prove a fundamental lemma that suggests many possible generalizations
of the above discussion of abelian girth and the Moore bound.
First we remark that the girth of a graph is smallest positive length of 
a cycle
that embeds in the graph.
Our fundamental lemma states that the abelian girth 
is the smallest number of
edges in a graph of Euler characteristic $-1$ that embeds in the graph,
provided that we weight the edges appropriately: 
namely, we count each
edge twice, except that in ``barbell graphs'' we count each edge in the
bar four times.

We remark that both girth and abelian girth can be viewed as linear
algebraic invariants of graphs;
indeed, girth 
is often studied as a property of the adjacency matrix of the graph 
(see, for example, \cite{amit}), and the abelian girth relates to sheaves
of vector spaces over the graph.
So our lemma fundamental to proving \eqref{eq:main1} also suggests that
there could be other such girth-type invariants, both
(1) arising as linear algebraically from the graphs, and
(2) satisfying inequalities analogous to \eqref{eq:main1}.

The rest of this paper is organized as follows.
In Section~2 we give some precise definitions and state our main
theorems.
In Section~3 we prove \eqref{eq:main2}.
In Section~4 we prove the fundamental lemma as well as
\eqref{eq:main1} which follows quickly from the fundamental lemma.
In Section~5 we describe the graphs $X^{p,q}$ from \cite{LPS} and show that
they obey Equation ~\eqref{eq:main4}.

\section{Main Results}

In this section we fix some terminology regarding graphs and
formally state our main results.

\subsection{Graph Terminology}

This entire subsection consists of definitions used throughout
this paper; they are more or less standard.

By a \emph{graph}
we shall mean a quadruple $G=(V_G,E_G,t_G,h_G)$
where $V_G$ and $E_G$ are sets 
(the ``\emph{vertices}'' and ``\emph{edges}'',
respectively) and $t_G,h_G \from E_G\to V_G$ are maps (the ``tail'' and
``head'' of each edge, respectively).
The {\em Euler characteristic} of $G$ is
$$
\chi(G) = |V_G| - |E_G| ,
$$
where $|\ \cdot\ |$ denotes the cardinality, provided that $G$ is finite,
i.e., that $|V_G|$ and $|E_G|$ are finite.
We define the {\em directed edge set} of $G$ to be
$$
D_G = E_G \times \{ +,-\},
$$
and we extend $h_G$ and $t_G$ to be functions on $D_G$ such that for $e\in E_G$
$$
h_G(e,+) = t_G(e,-) = h_G(e), \quad t_G(e,+) = h_G(e,-) = t_G(e);
$$
for $e \in E_G$ we say that $(e, +)$ and $(e, -)$ are $\emph{inverses}$ of each other.
We denote the inverse of $u\in D_G$ by $u^{-1}$.
A \emph{walk} of
\emph{length} $m$ in $G$ is
a sequence of directed edges
$$
w=(u_1,\ldots, u_m)
$$
such that the head of $u_i$ is the tail of $u_{i+1}$ for $i=1,\ldots,m-1$;
we define the vertices $t_G(u_1)$ and $h_G(u_m)$ to be
the \emph{starting} and \emph{terminating} vertices of $w$,
jointly the \emph{endpoints} of $w$, and the {\em length} of $w$ to be $m$,
denoted $l(w)$;
we say that $w$ is {\em closed} if its two endpoints are the same vertex;
we refer to the {\em edges} of $w$ as the $e\in E_G$ such that at
least one of $(e,+),(e,-)$ is an directed edge of $w$;
we say that $w$ is \emph{non-backtracking} if there is no $i=1,\ldots,m-1$
for which $u_i^{-1}=u_{i+1}$;
we say that $w$ is {\em strongly closed, non-backtracking} if $w$
is closed and non-backtracking, and $u_1,u_m$ are not inverses of each other;
we say that $w$ is a {\em path} (respectively, {\em cycle}) if $w$ is 
non-backtracking with distinct endpoints (respectively strongly closed,
non-backtracking),
and each edge $e\in E_G$ appears at most once in $w$ (i.e.
$(e,+)$ and $(e,-)$ appears at most once in $u_1,\ldots,u_m$).
The \emph{inverse} of  $w$ is defined as the walk 
$$w^{-1} = (u_m^{-1},\ldots, u_1^{-1}).$$

Given a walk $w=(u_1,\ldots, u_m)$ and a walk $k = (t_1, \ldots, t_n)$ such that
the terminating vertex of $w$ is the starting vertex of $k$,
we define the \emph{product} $wk$  to be the walk $(u_1,\ldots, u_m, t_1, \ldots, t_n)$.
We say that a walk $w$ as above {\em joins} its two endpoints.
We say that 
$G$ is \emph{connected} if any two of its vertices are joined by some walk
in $G$.

The contiguous appearance of an edge and its inverse in a walk is
called a {\em reversal}.
Each walk, $w$, can be {\em reduced} by successively discarding
its reversals; the walk
obtained is the {\em reduction} of $w$ and is known to be independent of the choice of
pairs to reduce; the reduction therefore contains no reversals and is non-backtracking \cite{dicks}.
We use the notation
$\red(w)$ for the reduction of $w$. 
Note that $\red(w^{-1}) = \red(w)^{-1}$ for any walk $w$ since
the portion of a walk that are removed by reduction remain the same when we 
take the inverse of that walk.

\begin{defn}
In each element $\omega$ of $\pi_1(G,u)$, there exists one unique non-backtracking walk in $G$;
any walk in $\omega$ reduces to that unique non-backtracking walk \cite{dicks}. By the \emph{length} of 
$\omega$, denoted $|\omega|$, we mean the length of the unique non-backtracking walk
in $\omega$. Note this is not word length with respect to a set of generators of $\pi_1(G,u)$. Similarly,
even though closed walks are members of homotopy classes in the fundamental group,
when we refer to the length of a walk or the reduction of a walk $w$  we mean length and reduction
as defined earlier and not in with respect to a set of generators for a free group. 

\end{defn}

As usual, a \emph{morphism} of graphs, $f\from G\to H$, is a pair
$f=(f_V,f_E)$ of maps $f_V\from V_G\to V_H$ and
$f_E\from E_G\to E_H$ such that
$t_H\circ f_E=f_V \circ t_G$ and $h_H\circ f_E=f_V \circ h_G$.
We often drop the subscripts from $f_V$ and $f_E$.
The set of morphisms
will be denoted $\Mor(G,H)$.  Thus morphisms are homomorphisms of the
underlying undirected graphs which preserve the orientation.

As usual, the {\em girth} of a graph is the length of its shortest closed,
non-backtracking walk; this length is necessarily positive by our conventions
above.  (In the literature one often allows for walks of length zero, which
we do not consider here.)

\subsection{Our Fundamental Lemma}

In this subsection we discuss our main results.

\begin{defn}
The {\em abelian girth} of a graph, $G$, denoted $\abl(G)$,
is the minimum $m\ge 1$ such that there
is a closed, non-backtracking walk
$$
w=(u_1,\ldots,u_m)
$$
such that each edge is traversed the same number of times in $w$ in
both directions, i.e., for each $e\in E_G$ the edge $(e,+)$ appears the same
number of times among $u_1,\ldots,u_m$ as does $(e,-)$.
\end{defn}

The abelian girth is the same as the girth of the {\em universal abelian
cover} of $G$; see \cite{friedman_hnc_memoirs}. Given a walk $w$ and an $e\in E_G$,
if $(e,+)$ appears $i_+$ times and
$(e,-)$ appears $i_-$ times
we say $e$ \emph{appears a net $i_+-i_-$ times} in $w$.
We refer to a walk as \emph{
edge neutral} if every edge appears a net $0$ times in it.

We begin with a fundamental lemma.
We remind the reader that our conventions insists that cycles and
paths are of positive length.
\begin{defn}
Let $G$ be a connected graph of Euler characteristic $-1$ without leaves,
i.e., without vertices of degree one.
We say that $G$ is 
\begin{enumerate}
\item a {\em figure-eight} graph if $G$ consists of
two cycles, mutually edge disjoint, sharing the same endpoint;
\item a {\em barbell} graph if $G$ consists of two cycles, $w_1,w_2$,
and one path 
$b$, all mutually edge disjoint, such that $b$ joins the endpoints 
of $w_1$ and $w_2$; we refer to
$b$ as the {\em bar} of $G$;
\item a {\em theta} graph if $G$ consists of two vertices joined by
three paths mutually edge disjoint.
\end{enumerate}
\end{defn}

It is well known that the above three cases classifies all connected
graphs of Euler characteristic $-1$ without leaves
(see, for example, \cite{linial_puder}.)

\begin{defn}
Let $G$ be a connected graph of Euler characteristic $-1$ without leaves.
We define the \emph{abelian length} of $G$, denoted $l_{\abl}(G)$,
to be twice its number of edges except that each edge of its bar
(if $G$ is a figure-eight graph) is counted four times.
\end{defn}

\begin{rmk}
Observe that in each case, the abelian length of such $G$ is its abelian
girth.  This is clear for the figure-eight and the theta (traverse each cycle twice) and not hard for the barbell (must traverse each cycle twice,
in reverse orientations, so must traverse the bar between each of the four
cycle traversals).
\end{rmk}

\begin{lem}\label{lem:fundamental}\emph{(The Fundamental Lemma)}
For any graph, $G$, we have
$$
\abl(G) = \min_{G'\subset G} l_{\abl}(G')
$$
taken over all subgraphs, $G'$, that are connected, of Euler characteristic
$-1$, and without leaves.
If no such $G'$ exist, then the above minimum is taken to be infinity,
as is the abelian girth of $G$, and there are no closed non-backtracking
walks that traverse each edge the same number of times in both directions.
\end{lem}

By the observation above, the upper bound is clear, and the content is in
the lower bound.

\subsection{Main Results in this Paper}

Now we can easily state the other main results in this paper.

\begin{thm}\label{thm:girths}
For any graph we have
$$
\gir(G) \le \abl(G)/3.
$$
\end{thm}
This theorem is an easy corollary of Lemma~\ref{lem:fundamental}.

\begin{thm}\label{thm:moore}
For any fixed $d$, let $G_{n,d}$ be any $d$-regular graph on $n$ vertices.
Then for large $n$ we have
$$
\abl(G) \le 6 \log_{d-1} n + o_n(1).
$$
\end{thm}

Our last theorem uses the \cite{LPS} graphs $X^{p,q}$ whose definition
we save for Section~\ref{sec:lps}.

\begin{thm}\label{thm:lps}
Let $p,q\equiv 1 \pmod 4$ be prime with $(q/p)=1$.  
The graph $X^{p,q}$ of \cite{LPS} of our Definition~\ref{defn:lps}
has degree $d=p+1$ and $n=q(q^2+1)$, and for fixed $p$ and large $q$
(i.e., fixed $d$ and large $n$) we have
$$
\abl(X^{p,q}) \le (16/3) \log_{d-1}n + o_n(1).
$$
\end{thm}

\section{Proof of Theorem~\ref{thm:moore}}

In this section we prove Theorem~\ref{thm:moore}.
First we state a few well-known facts regarding free groups and graphs.



\begin{lem}
Let $G$ be a finite graph.  Then the fundamental group, $\pi_1(G,u)$, 
of homotopy classes of closed walks in $G$ about $u$ for any $u\in V_G$
is
isomorphic a free group on a finite set of generators.
\end{lem}

The following fact is well known.  For ease of reading we provide a 
proof which is also found in texts on group theory such as \cite{lyndoncom}.

\begin{lem}\label{lem:commute}
Let $F$ be a free group and let $\alpha$ and $\beta$ be two elements of $F$ that commute with each other.
Then there exists an $\omega \in F$ such 
that $\alpha=\omega^m$ and $\beta=\omega^{m'}$ for $m, m' \in \Z$.
\end{lem}
\begin{proof}
By the Nielsen-Schreier theorem,
the subgroup that $\alpha,\beta$ generate is a free group.
And yet, any two elements of this subgroup commute, and hence this subgroup
cannot be free and rank more than one.  Hence this subgroup is generated
by some $\omega\in F$, and the conclusion follows.
\end{proof}

For any closed walk $w$, we use the notation $[w]$ for the homotopy class in $\pi_1(G,u)$  that has
$w$ as a representative.
Suppose $a$ and $b$ are closed walks and [a] and [b] commute in $\pi_1(G,u)$. Note that equality
in the previous lemma is with regards to $\pi_1(G,u)$, meaning an equality of homotopy classes but not 
necessarily walks. 
Reducing all the elements of a homotopy class gives one unique walk though.
So since $[a]$ and $[b]$ commute, we have $\red(a)=\red(w^m)$ and $\red(b)=\red(w^{m'})$ for 
$m, m' \in \Z$ and $w$ a closed walk. These equalities may not exist without the reductions.  We may assume $w$
is reduced, since $\red(w^m) = \red( \red(w)^m).$

The following lemma is also well known and easy.
\begin{lem}\label{lem:unequal_length}
Let $w$ be a closed nontrivial non-backtracking walk in a graph.
Then if for integers $m,m'$ we have that the length of $\red(w^m)$ equals that
of $\red(w^{m'})$, then $m=\pm m'$. 
\end{lem}
\begin{proof}
Let $y$ be a maximal length walk such that $w=yxy^{-1}$ for some walk $x$.  Then it
is easy to check that for $m\ne 0$
the length of $\red(w^m)$ with respect to $S$ is precisely
$$
2 |y| + |m| |x|.
$$
\end{proof}

\begin{proof}[Proof of Theorem~\ref{thm:moore}]
Fix any vertex, $v\in V_G$.
Then for any integer $h\ge 1$, there are $d(d-1)^{h-1}$ non-backtracking
walks of length $h$ from $v$.
So let
$h$ be the smallest integer for which
$$
d(d-1)^{h-1} \ge 2n+1;
$$
then, by the pigeon hole principle,
there are three distinct non-backtracking walks, $a,b,c$, in $G$ beginning
in $v$ and terminating in the same vertex $u$ (which may or may not
equal $v$).  We remark that
$$
h  = \log_{d-1}n + o_n(1).
$$
Hence, to prove the theorem it suffices to show that 
we have
\begin{equation}\label{eq:bound_ab_h}
\abl(G) \le 6h .
\end{equation}

\newcommand{\newd}{\ell}

The rough idea is simple.
We break the analysis into two cases: $u=v$ and $u\ne v$.
In either case we define new walk, $\newd$, based on $a,b,c$
such that (1) the length of
$\newd$ is $4h$ or $6h$ (in the respective two cases), 
(2) each edge of $E_G$ appears a net $0$ times in $\newd$, but (3) $\newd$ is not necessarily non-backtracking.
The main work is to show that $\newd$ does not reduce to the empty word.
Discarding a consecutive pair of an (unoriented) edge of $\newd$ and its
opposite retains the property that each edge is traversed the same 
number of times in each direction.
Hence if $\newd$ is reduced to a non-empty, non-backtracking walk, $\newd'$, then
$\newd'$ is edge neutral, and so
$$
\abl(G)\le 6h.
$$

Let us describe $\newd$ as above.
If $u=v$, then at least two of $a,b,c$ are not inverses of each other;
if $a,b$ are such walks, then we set $\newd=aba^{-1}b^{-1}$. 
This 
walk is of length $4h$.
If $u\ne v$, then we set $\newd$ to $\newd=ab^{-1}ca^{-1}bc^{-1}$, which is s walk
of length $6h$. For any $e \in E_G$, if $e$ appears a net $k$ times in a walk 
$w$ then $e$ appears a net $-k$ times in the walk $w^{-1}$. Thus $\newd$ in either case is
edge neutral. 

  It remains to show that $\newd$ reduces to a 
non-empty non-backtracking walk.

The case $u=v$ and $\newd=aba^{-1}b^{-1}$ where $a$ and $b$ (are distinct and)
are not inverses of each other is relatively easy.
If $\newd$ reduces to the empty word $e$ and if $\alpha = [a]$, $\beta = [b]$ and $\delta = [\newd]$ then
$\delta= [e] = [\alpha, \beta]_\mathrm{comm}$ where $[\ ,\ ]_\mathrm{comm}$ denotes the commutator. 
So
$\alpha$ and $\beta$ commute as elements of $\pi_1(G,u)$.
Hence, by Lemma~\ref{lem:commute},
and since $\pi_1(G,u)$ is a free group,
it follows that $a=\red(w^m)$ and $b=\red(w^{m'})$ for some closed walk $w$ with $w\ne 1$;
but then by Lemma~\ref{lem:unequal_length}, since $a$ and $b$ have the
same length, we have $m=\pm m'$, which contradicts the fact that $a$ and 
$b$ are distinct and not inverses of each other.

Now suppose $u \neq v$.  
We remark that as elements of $\pi_1(G,u)$ we have
$$
\newd=ab^{-1}cb^{-1}ba^{-1}bc^{-1}= [ab^{-1}, cb^{-1}],
$$
where we use $[f_1,f_2]$ to refer to the walk $f_1f_2f_1^{-1}f_2^{-1}$
for any $f_1$ and $f_2$ that are closed walks about the same vertex.
Let $\delta = [\newd]$, $\alpha = [ab^{-1}]$ and $\gamma = [cb^{-1}]$ and let $e$ be the trivial walk.
So if $\newd\in [e]$, we have that $1 = \delta = [\alpha, \gamma]_\mathrm{comm}$ and so
 $[ab^{-1}]$ and $[cb^{-1}]$ commute in the fundamental group.
Hence $\red(ab^{-1})=\red(w^m)$ and $\red(cb^{-1})=\red(w^{m'})$ for some closed non-backtracking
walk $w$.
Let $w= yx y^{-1}$ with
$x$ a closed non-backtracking walk and $y$ a walk of maximal length; 
then (1) $|x|>0$ as $|x|=0$ would imply
$a=b$, (2) the first and last
edges of $x$ are not inverses of each other, and (3)
$\red(w^m)=yx^m y^{-1}$.

Next we want to make some remarks on $ab^{-1}$ and $w^m$ based on the 
fact that
$\red(ab^{-1})=\red(w^m)$; we will later repeat similar remarks for $cb^{-1}$ and $w^{m'}$
based on the fact that $\red(cb^{-1})=\red(w^{m'}$).
Note 
$$a=\red(ab^{-1}b) = \red( \red(ab^{-1})b) = \red(\red(w^m)b) = \red(w^mb)$$
as reductions are independent of the order in which we reduce portions of the word.

Let $p$ be the maximal prefix of $b$ whose inverse is a suffix of
$\red(w^m)= yx^m y^{-1}$ (a priori $p$
could be as long as the shorter of $yxy^{-1}$ and $b$); since
$$
a=\red(w^m b),
$$
we have
$$
|a|= 2|y|+|m|\,| x | +|b|- 2|p|.
$$
But by assumption $|a|=|b|$, and hence
$$
2|y|+|m|\,| x|=2|p|.
$$
It follows that
\begin{equation}\label{eq:Plength}
|p| = |y| + |m|\,| x |/2 > |y|,
\end{equation}
since $|x|> 0$ and $m\ne 0$.
Since $|p|>|y|$ and $p$ is a suffix of $w^m=yx^m y^{-1}$, 
it follows that $p$ begins with $y$ and
contains at least one more edge; hence
$p=y x_1$ where $|x_1|>0$ and the first edge of $x_1$ is the inverse of the
last edge of $x^m$.

However, since
$\red(cb^{-1})=\red(w^{m'})$, the very same arguments show that if $p'$ is defined
analogously, i.e., as the 
maximal prefix of $b$ that is a suffix of $\red(w^{m'})$, then
$p'=y x_1'$ where $|x_1'|>0$ and the first letter of $x_1'$ is the
inverse of the last letter of $x^{m'}$.
But since both $p$ and $p'$ are both prefixes of $b$, we must have that
$m$ and $m'$ have the same sign: for if, say, $m>0$ and $m'<0$, then
the last edge of $x^m$ or it's orientation is different than the last edge of $x^{m'}$
by the maximality of $y$ in the equation $w=yxy^{-1}$.

Hence
$$
a = \red(w^m b), \quad c=\red(w^{m'}b),
$$
$a,b,c$ have the same length, and without loss of generality
we may assume $m>m'>0$.
But then \eqref{eq:Plength} and its analogue for $cb^{-1}=w^{m'}$ show that
$$
m = (|p|-|y|)/| x |, \quad\mbox{and}\quad
m' = (|p'|-|y|)/| x |.
$$
It follows that $|p|>|p'|$.  Since both $p$ and $p'$ are prefixes of
$b$, we have that $p'$ is a prefix of $p$.  But $p'$ is the maximal
suffix of $w^{m'}$ that is also
 a prefix of $b$; since $w^{m'}$ is also a suffix of  $w^{m}$  and $w^m$ is a prefix of $b$ it
follows that $p'=w^{m'}$ (otherwise $p$ could not be larger than $p'$).
But in this case
$$
|c| = |b|-|p'| < |b|,
$$
which contradicts that fact that $|c|=|b|$ from our choice of $a$, $b$ and $c$.
\end{proof}

\begin{lem}\label{lem:scholium}
If a graph $G$ has three distinct closed, non-backtracking walks from a vertex to itself and each walk 
has length $l$ then $\abl(G) \leq 4l.$
\end{lem}
\begin{proof}
This was shown to be true in our proof of Theorem~\ref{thm:moore}. See the discussion for the case $u=v$.
\end{proof}

\section{Proofs of the Fundamental Lemma and Theorem~\ref{thm:girths}}



\begin{proof}[Proof of Lemma \ref{lem:fundamental}]
Let $w$ be an edge neutral, 
closed, non-backtracking walk of minimal (positive) length,
and let $B$ be the subgraph of $G$ of vertices and edges that occur
in $w$.  Then $B$ is a connected subgraph of $G$ such that each vertex
has degree at least two.  
If every vertex in $B$ has degree exactly two, then $B$ would be a 
cycle, which is impossible any non-backtracking walk in a cycle traverses
edges in at most one direction.
Hence at least some vertex of $B$ is strictly greater than two;
hence the formula
$$
\chi(B) = \sum_{v\in V_{B}} \bigl( 2 - \deg(v) \bigr)
$$
shows that $\chi(B)\ge -1$.

Now $w$ traverses each edge of $B$ at least once in each direction, and
hence
$$
l(w) \ge 2\, |E_B|.
$$

Recall that the abelian length of a figure-eight, theta or barbell graph
was defined as its abelian girth.  In particular, if $B$ is such a graph
we are done.  Moreover, in general it suffices to show that there is another
walk $w'$ which is edge neutral and closed non-backtracking for which
$l(w)\ge l(w')$ and which is supported on a subgraph $B'$ of $B$ of
Euler characteristic $-1$.  We therefore assume by contradiction that
no such $w'$ exists.

It follows immediately that $B$ cannot contain a theta or figure-eight
graph $B'$; indeed 
$$
l(w) \ge 2|E_B|;
$$
while for such $B'$, $Abl(B') = 2|E_{B'}|$ and $B'$ would be a proper subgraph
of $B$, hence have fewer edges.  This would contradict the minimality of $B$.

Second, we claim that $B$ cannot be $1$-connected (have an edge $e$ whose
removal would disconnected $B$).  Indeed, if there is such an edge, we can
extend in both directions to a path until we encounter vertices $v_1,v_2$
of degree greater than $2$.  Call this path $p$.

Let $B_1, B_2$ be the components we get by removing $p$ from $B$ (but keeping
$v_1 \in B_1$ and $v_2 \in B_2$).  Then any non-backtracking closed walk in $B$
must follow a traversal of $p$ from $B_1$ to $B_2$ by a non-trivial closed
walk in $B_2$, a traversal of $p$, and then a non-trivial walk in $B_1$.
In particular, each of $B_1,B_2$ must contain a cycle and $w$ must traverse
$p$ at least $4$ times, between successive visit to $B_1,B_2,B_1,B_2$.

Each visit to $B_i$ is a closed non-backtracking walk in $B_i$ beginning in
$v_i$. Accordingly let $c_i$ be the shortest such walk.  Then we may assume
$w$ is the sequence $c_1,p,c_2,p^{-1},c_1^{-1},p,c_2^{-1},p$ (this is no
longer than $w$), and in particular that $B_i$ is the support of a shortest
closed cycle.  But then $B_i$ is unicyclic, and since it has no leaves
it must  be a "lollipop": a simple cycle connected to $v_i$ by a "stick"
(simple path).  This which makes $B$ into a barbell graph and $w$ into its
shortest balanced closed non-backtracking walk.

Third, we claim that either the first or second claim must be
contradicted; this will complete the proof of the lemma.
Indeed, $B$ is not a tree, and therefore must contain a cycle, $C$.
The cycle must contain a vertex, $v$, of degree at least three, or else
$B$ would consist entirely of $C$, which would then have Euler characteristic
zero.
Let $e$ be any edge not in $C$ that is incident upon $v$, and let
$u$ be the other endpoint of $e$.
Since removing $e$ does not disconnect $B$, there must be a walk from
$u$ to a vertex of $C$ that does not contain $e$; let $p$ be such a
walk of minimal length.  Then $p$ begins in $u$ and is a beaded path
to a vertex of $c$.  But then the path $e$ followed by $p$ is
disjoint from $C$ and joins $v$ to another vertex of $C$, which yields
either a figure-eight graph (if $p$ terminates in $v$) or a
theta graph (if $p$ terminates in a vertex of $C$ that is not $v$).

Hence our first claim, that $B$ does not contain a theta or figure-eight
subgraph, is violated.
\end{proof}

We now prove Theorem~\ref{thm:girths} as a consequence of the 
Fundamental Lemma.

\begin{proof}[Proof of Theorem~\ref{thm:girths}.]
Let $G'\in \mathcal{C}$ be the subgraph of $G$ of minimal abelian length.
If $G'$ is homeomorphic to the barbell graph, then $\abl(G) \geq 4+4\gir(G)$
since each of the two simple cycles in $G'$ is at least as long as the girth.
Similarly, $\abl(G) \geq 4\gir(G)$ if $G'$ is homeomorphic to a figure-eight graph.  In the case of a theta graph, we have
two vertices $u$ and $v$ in $G'$ of degree $3$ and three simple paths from $u$
to $v$. Let the lengths of the simple path be $l$, $m$ and $n$.
Then $l+m,m+n,n+l \geq \gir(G)$.  Adding these inequalities gives
$$3\gir(G) \leq 2(l+m+n) = 2|E_{G'}|=\abl(G)$$.
\end{proof}


\section{Abelian Girth of the LPS Expanders}
\label{sec:lps}


The Ramanujan graphs of Lubotzky, Phillips and Sarnak \cite{LPS} are an
infinite family of graphs with the largest known asymptotic girth.  As
mentioned earlier, these 
are graphs $X^{p,q}$, for primes $p,q\equiv 1 \pmod{4}$,
that are $d=p+1$ regular on $n=q(q^2+1)$ vertices for which 
$$
{\rm girth}(X^{p,q}) = (4/3) \log_{d-1} n + o_n(1).
$$
 This immediately leads to the bound on abelian girth 
$$\abl(X^{p,q}) \geq 4 \log_{d-1} n + o_n(1).$$
We will describe these Ramanujan graphs and obtain an upper bound on their
abelian girth in order to show that their abelian girth isn't so large that it
would be impossible to improve Theorem \ref{thm:moore}.

\begin{defn}
\label{defn:lps}
Let $p$ and $q$ be unequal primes congruent to $1$ mod $4$ with $(p/q) = -1$ and $q> \sqrt{p}.$ The integral quaternions, denoted $H(\Z)$, are given by
$$H(\Z) = \{\alpha = \alpha_0+\alpha_1 \mathbf{i} + \alpha_2 \jq +\alpha_3 \kq | a_j \in \Z\}.$$
We denote the conjugate of $\alpha$ by $\overline{\alpha}$ and define 
$N(\alpha) = \alpha \overline{\alpha}.$ Let $S$ be the set of all $\alpha$ in $H(\Z)$ satisfying
$N(\alpha)=p, \alpha \equiv 1 (\mod 2)$ and $\alpha_0 \geq 0$. It can be shown that $|S| = p+1.$ Define 
$\Lambda'(2)$ as the set of $\alpha \in H(\Z)$ such that $N(\alpha)=p^v$ for some
non-negative integer $v$ and $\alpha \equiv 1 (\mod 2).$ Define $\Lambda(2)$ as equivalence classes
 that identify  $\alpha$ and $\beta$ $\in H(\Z)$ if $\pm p^v_1 \alpha = p^v_2 \beta$ for some
$v_1, v_2 \in \Z.$ It is known that the Cayley graph of $\Lambda(2)$ is the $(p+1)$-regular tree.
Define $\Lambda(2q)$ by
$$ \Lambda(2q) = \{[\alpha ] \in \Lambda(2)| 2q \textrm{ divides } \alpha_j, j = 1,2,3\}.$$
This is a normal subgroup of $\Lambda(2)$ and $X^{p, q}$, the \emph{LPS graph}, is defined as the Cayley graph of
$\Lambda(2)/\Lambda(2q)$ with generators $S/\Lambda(2q)$. This graph is known to be a $(p+1)$-regular bipartite graph on $q(q^2+1)$ vertices.
\end{defn}

Since $X^{p, q}$ is a Cayley graph it is vertex transitive which allows us to assume its smallest cycle goes from the
identity element of $\Lambda(2)$ to some nontrivial element of $\Lambda(2q)$ along the infinite Cayley graph of
$\Lambda(2)$. We now define a vertex' depth as its distance from the identity in the Cayley graph of $\Lambda(2)$. Biggs and Boshier \cite{BB} show that if $[b] \in \Lambda(2)$ is at depth
$2r$ and $r>0$ then there is some $b$ in the equivalence class $[b]$ such that
$$b_0 = \pm(p^r - mq^2)$$
with $m>0$ and even.

\begin{defn}
Positive integers are called \emph{good} if they are not of the form $4^\alpha(8\beta+7)$ for integers $\alpha, \beta \geq 0$. \end{defn}
The following is Lemma 2 of \cite{BB}.
\begin{lem}
There exists a $[b] \in \Lambda(2q)$ at level $2r$ with $b_0 = p^r - mq^2$ with $m>0$ and $b_0$ positive
if and only if $2mp^r-m^2q^2$ is good.
\end{lem}
 In the paragraph before this Lemma, Biggs and Boshier prove at least one of
 the integers $2mp^r-m^2q^2$
is good for the cases $m=2$ and $m=4$ so long as they are both positive.

The following lemma is original.
\begin{lem}
If $m=4 +8c$ for nonnegative integers $c$, then $2mp^r-m^2q^2$ is good if it is positive.
\end{lem}
\begin{proof}
Note
$$\frac{2mp^r-m^2q^2}{4} = (2+4c)p^r - (4 + 16c+ 16c^2)q^2 \equiv 2 \quad \bmod{4}$$
implying that $2mp^r-m^2q^2$ is good.
\end{proof}
So for $m = 4, 12, 20,$  $2mp^r-m^2q^2$ is good if we can show it is positive.
 Let $r_0$ be the smallest positive integer such that $p^{r_0} > 10q^2$, which makes $2mp^{r_0}-m^2q^2=m(2p^{r_0}-mq^2)$ positive  for  $m<20$.
Then 
$p^{r_0 -1} < 10q^2$ (note this is a strict inequality since $p^{r_0 -1}$ is clearly not a divisible by $10$) and so 
$$ r_0 < 2\log_p q + \log_p 10 + 1.$$ Thus there exists three distinct 
$[b] \in \Lambda(2q)$, since each value of $m$ produces a different $b_0$, which means there are three distinct non-backtracking closed walks of length $2r_0$ from the identity vertex to itself. From Lemma~\ref{lem:scholium} we showed that this situation would imply
the abelian girth of $X^{p,q}$ is at most $8r_0.$
Since $n=q(q^2+1)$ and $d-1=p$
 we have   
\begin{equation}\label{eq:lpsbound}
\abl(X^{p,q}) \leq \frac{16}{3} \log_{d-1} n (1+o(1))
\end{equation}
for $p$ fixed and $q$ large.

The $16/3$ constant above may not be the optimal constant. In our arguments,
we found three closed walks of with a single starting vertex and length
$2r_0$.  If instead we had found three walks of length $r_0$
sharing distinct starting and terminating vertices the argument would have
given the constant of $4$ in \eqref{eq:lpsbound}.
Note that any closed walk from a vertex $v$ to itself of length $2r_0$
already gives two distinct walks of length $r_0$ from $v$ to the $m$, the
middle vertex of the cycle. Identifying one more walk from $v$ to $m$ of
length $r_0$  would be sufficient to improve the $16/3$ coefficient.



\end{document}